\documentclass[10pt]{article}
\usepackage{amsthm,amsfonts,amsmath,amscd,amssymb,verbatim}
\usepackage{graphicx}
\title{A note on the stationary Euler equations of hydrodynamics}
\author{K.~Cieliebak
and E.~Volkov}
\date{}
\theoremstyle{plain}
\newtheorem{theorem}{Theorem}[section]
\newtheorem{thm}[theorem]{Theorem}

\newtheorem{cor}[theorem]{Corollary}

\newtheorem{prop}[theorem]{Proposition}
\newtheorem{lemma}[theorem]{Lemma}

\theoremstyle{remark}

\newtheorem{remark}[theorem]{Remark}
\newtheorem{example}[theorem]{Example}

\newcommand{\ol}{\overline}

\newcommand{\p}{\partial}

\newcommand{\om}{\omega}

\newcommand{\eps}{\varepsilon}

\newcommand{\la}{\langle}
\newcommand{\ra}{\rangle}
\newcommand{\wt}{\widetilde}
\newcommand{\wh}{\widehat}
\newcommand{\N}{{\mathbb{N}}}
\newcommand{\Z}{{\mathbb{Z}}}
\newcommand{\R}{{\mathbb{R}}}
\newcommand{\C}{{\mathbb{C}}}

\newcommand{\HHH}{{\mathbb{H}}}
\newcommand{\curl}{{\rm curl\,}}

\renewcommand{\min}{{\rm min}}
\renewcommand{\max}{{\rm max}}

\newcommand{\dist}{{\rm dist}}

\newcommand{\FF}{\mathcal{F}}

\newcommand{\fm}{{\mathfrak m}}
\begin{document}

\maketitle

\begin{abstract}
\noindent
This note concerns stationary solutions of the Euler equations for an
ideal fluid on a closed 3-manifold. We prove that if the velocity
field of such a solution has no zeroes and real analytic Bernoulli
function, then it can be rescaled to the Reeb vector field of a
stable Hamiltonian structure. In particular, such a vector field has a
periodic orbit unless the 3-manifold is a torus bundle over the
circle. We provide a counterexample showing that the correspondence
breaks down without the real analyticity hypothesis.  
\end{abstract}

\section{Introduction}\label{sec:intro}

The time evolution of an incompressible, inviscous (``ideal'') fluid
of constant density is described by the {\em Euler equations} 
$$
   \partial_tX+\nabla_XX = -\nabla p,\qquad {\rm div}\,X=0
$$
for the velocity field $X$ of the fluid and its pressure $p$ (both
time-dependent). These equations make sense on any Riemannian $3$-manifold 
$(M,g)$ equipped with a volume form $\mu$ (not necessarily the one
induced by the metric). Note that the first equation involves only the
metric (via the covariant derivative and the gradient), while the
second one involves only the volume form (via the divergence defined
by $L_X\mu=({\rm div}\,X)\mu$).  

%

In this note we are interested in stationary solutions of the Euler
equations, i.e., time-independent pairs $(X,p)$ satisfying the {\em
  stationary Euler equations}
\begin{equation}\label{eq1}
   \nabla_XX=-\nabla p,
\end{equation}
\begin{equation}\label{eq2}
   L_X\mu=0.
\end{equation}
Since by equation~\eqref{eq1} the velocity field $X$ determines the
pressure $p$ uniquely up to a constant, we will often suppress
explicit mentioning of $p$ and refer to $X$ as a stationary solution
of the Euler equations. 
Equation~\eqref{eq1} can equivalently be written as (see
e.g.~\cite{AK98}) 
\begin{equation}\label{eq3}
   \curl X\times X = -\nabla h,
\end{equation}
where $h:=p+\frac{|X|^2}{2}$ is the {\em Bernoulli function}, and the
curl and cross product of vector fields are defined by
$d(i_Xg) = i_{{\rm curl}\,X}\mu$ and $i_{X\times Y}g = i_Yi_X\mu$. 
Another equivalent way of writing equation~\eqref{eq1}, which will be
particularly useful for our purposes, is in terms of
the $1$-form $\lambda:=i_Xg$ as (see e.g.~\cite{CV10}) 
\begin{equation}\label{eq4}
　　　i_Xd\lambda = -dh.
\end{equation}
Note also that equation~\eqref{eq2} is equivalent to closedness of the
$2$-form $\om:=i_X\mu$. 
The goal of this note is to prove the following result. 

\begin{thm}\label{thm:real-anal-stable}
Let $X$ be a solution of the stationary Euler
equations~\eqref{eq1},~\eqref{eq2} with respect to a metric $g$ and
volume form $\mu$ on a closed, oriented $3$-manifold $M$. 
Suppose that $X$ has no zeroes and its Bernoulli function is real
analytic (for some real analytic structure on $M$). Then $X$ also
solves equation~\eqref{eq1} with respect to a different metric $\wh g$
such that the corresponding Bernoulli function is constant. 
\end{thm}

Let us discuss some consequences of this results.

(1) Following earlier work of Etnyre and Ghrist~\cite{EG00a}, it was
observed in~\cite{CV10} that a vector field as in
Theorem~\ref{thm:real-anal-stable} has the following symplectic
interpretation. The nowhere vanishing closed $2$-form $\om$, the
$1$-form $\wt\lambda=i_X\wh g$ and the rescaled vector field $\wt
X=X/\wt\lambda(X)$ satisfy
\begin{equation}\label{eq:SHS}
　　　i_{\wt X}d\wt\lambda = i_{\wt X}\om = 0,\qquad \wt\lambda(\wt X)=1.
\end{equation}
In symplectic terminology~\cite{BEHWZ03}, this means that
$(\om,\wt\lambda)$ is a {\em stable Hamiltonian structure with Reeb
vector field $\wt X$}. So we have shown

\begin{cor}\label{cor:real-anal-stable}
A vector field $X$ as in Theorem~\ref{thm:real-anal-stable} can be
rescaled by a positive function to the Reeb vector field of some stable
Hamiltonian structure. 
\end{cor}

(2) By Corollary~\ref{cor:real-anal-stable}, the $1$-dimensional
oriented foliation defined by $X$ has the 
same dynamical properties as the foliation defined by the Reeb vector
field of some stable Hamiltonian structure. In particular, the proof
of the Weinstein conjecture for stable Hamiltonian structures by
Hutchings and Taubes~\cite{HT09} yields

\begin{cor}\label{cor:real-anal-per}
Let $M$ be a closed, oriented $3$-manifold which is not a $2$-torus
bundle over the circle. Then every stationary solution of the Euler
equations on $M$ with real analytic Bernoulli function possesses a
zero or a periodic orbit. 
\end{cor}

This result was previously proved by Etnyre and Ghrist~\cite{EG00a} in the
case of the $3$-sphere, and combining their arguments with~\cite{HT09}
yields a simple direct proof of Corollary~\ref{cor:real-anal-per}; see
Section~\ref{sec:sing} below. 
\medskip

(3) Let $\wt\lambda$, $\wt X$ be as in (1) and define a new metric
$\wt g:=\wt\lambda(X)\wh g$. Then $\wt\lambda=i_{\wt X}\wt g$, i.e.,
$\wt\lambda$ is the dual $1$-form to $\wt X$ with respect to the
metric $\wt g$. Equation~\eqref{eq:SHS} implies that $\wt X$ satisfies
equation~\eqref{eq3} with respect to the metric $\wt g$ and constant
Bernoulli function $\wt h$, and equation~\eqref{eq2} with respect to
the volume form $\wt\mu=\wt\lambda(X)\mu$. Moreover, the corresponding
pressure $\wt p=\wt h-\wt g(\wt X,\wt X)/2$ is constant, so we have shown

\begin{cor}\label{cor:real-anal-pressure}
A vector field $X$ as in Theorem~\ref{thm:real-anal-stable} can be
rescaled by a positive function to a solution of the stationary Euler
equations, with respect to some new metric and volume form, such that
the corresponding pressure is constant. 
\end{cor}

Geometrically, the equations $\wt\nabla_{\wt X}\wt X=0$ and $\wt g(\wt
X,\wt X)=1$ for the rescaled vector field $\wt X$ and the metric $\wt
g$ mean that all orbits of $\wt X$ are geodesics parametrized by
arclength; see~\cite{Su78} for further discussion of the question of
geodesibility. 

Corollary~\ref{cor:real-anal-pressure} somewhat resembles the
description of solutions of a mechanical system with Lagrangian
$L=|\dot q|^2/2-V(q)$ and energy $E>\max\,V$ as geodesics of the Jacobi
metric $\wt g=\sqrt{E-V(q)}g$. It would be nice to better understand
its physical interpretation. 
\medskip

(4) Corollary~\ref{cor:real-anal-stable} 
opens up the possibility of studying the dynamics of stationary
solutions of the Euler equations by holomorphic curve techniques as
in~\cite{HWZ98,HT09}, with implications for questions of hydrodynamical
stability. To illustrate this line of argument, let us recall an
instability criterion of Friedlander and Vishik~\cite{FV92}: If a
stationary solution $X$ of the Euler equations possesses a hyperbolic
zero or periodic orbit, then it is linearly unstable in the sense that
the linearized Euler equations at $X$ have solutions whose $L^2$-norm
grows exponentially in time. On the other hand, Hutchings and
Taubes~\cite{HT09} proved that each Reeb vector field with nondegenerate
periodic orbits on a closed $3$-manifold which is not a lens space has
a hyperbolc periodic orbit. Combining these two results with the
arguments in~\cite{EG05} in the case of the $3$-torus, we obtain 

\begin{cor}\label{cor:lin-unstable}
For a generic metric on a closed $3$-manifold with is not a lens
space, each curl eigenfield $\curl X=\lambda X$ with $\lambda\neq 0$
has a hyperbolic zero or periodic orbit, and is therefore linearly
unstable. 
\end{cor}

We hope that the techniques developed in this note can lead to similar
instability results for more general stationary solutions of the Euler
equations. 
\medskip

(5) Theorem~\ref{thm:real-anal-stable} and its
Corollaries~\ref{cor:real-anal-stable}
and~\ref{cor:real-anal-pressure} become false without the real
analyticity hypothesis: 

\begin{prop}\label{prop:counterex}
There exists a smooth, nowhere vanishing vector field $X$ solving the
stationary Euler equations~\eqref{eq1},~\eqref{eq2} on some closed,
oriented $3$-manifold $M$ which {\em cannot} be rescaled by a positive
function to the Reeb vector field of a stable Hamiltonian structure. 
\end{prop}

The proof of Theorem~\ref{thm:real-anal-stable} is based on the
observation that equation~\eqref{eq4} does not explicitly involve the
metric. If $(X,\lambda,h)$ is a solution of this equation with
$\lambda(X)>0$, then the vector field $X$ solves equation~\eqref{eq3}
with respect to any metric $g$ with $i_Xg=\lambda$. Such a
metric exists but is not unique because it can be defined arbitrarily
on $\ker\lambda$. So to prove Theorem~\ref{thm:real-anal-stable}, it
suffices to find a {\em stabilizing $1$-form}, i.e., a $1$-form
$\nu$ satisfying 
$$
   i_Xd\nu=0,\qquad \nu(X)>0. 
$$
Using the fact that $h$ is preserved by the flow of $X$ (which an
immediate consequence of equation~\eqref{eq4}), we will construct
$\nu$ in two steps. In Section~\ref{sec:sing} we construct $\nu$
near the singular level sets of $h$, and in Section~\ref{sec:reg} we 
extend it over the regular level sets. 
Proposition~\ref{prop:counterex} will be proved in
Section~\ref{sec:counterex}. 
\medskip

{\bf Acknowledgements. }
We thank M.~Goresky for a helpful discussion on real analytic sets. 
Much of this research was carried out during the first author's
visits to Chuo University, the Simons Center for Geometry and
Physics, and the Institute for Advanced Study. He thanks them for
their hospitality.

\section{Construction of $\nu$ near the singular level sets}\label{sec:sing}

We begin with some properties of real analytic sets in the plane. 
Let $D\subset\R^2$ be the open unit disk and $h:D\to\R$ the
restriction of a nonconstant real analytic function defined on a
neighbourhood of $\bar D$. Let $E\subset D$ be a connected component
of $h^{-1}(0)$ containing more than one point. 
For $x\in E$ we denote by $o(x)\geq 1$ the vanishing order of
$h$ at $x$, i.e., the lowest order of a nonvanishing partial
derivative of $h$ at $x$. Set $k:=\min\{o(x)\mid x\in E\}$. 

\begin{lemma}\label{lem:real-anal-sets}
(a) $E_1:=\{x\in E\mid o(x)>k\}$ is a finite set and
$E_0:=E\setminus E_1$ is a finite union of embedded curves.  

(b) Each $x\in E_1$ has an open neighbourhood $U\subset D$ such that
$E\cap U$ is a finite union $A_1\cup\cdots\cup A_m$ of $C^1$-embedded
half-open intervals that meet at the $x$ and are disjoint away from
$x$.  

(c) Near each $x\in E_0$ (after possibly replacing $h$ by $-h$ if $k$
is even) the function $g=\sqrt[k]{h}$ is real analytic with $d_xg\neq
0$, and the level sets of $h$ form a foliation near $x$.  
\end{lemma}

\begin{proof}
$E$ is a $C$-analytic set in the sense of~\cite{BH59}. 
Its $C$-rank of $E$ at $x\in E$ is defined to be the maximal number of
elements in the local vanishing ideal at $x$ (consisting of the germs
of holomorphic functions vanishing on $E^*$) of the minimal complex
analytic set $E^*$ containing $E$ whose differentials at $x$ are
linearly independent. By definition, the $C$-rank can only take the
values $0,1,2$. 

If the $C$-rank of $E$ at $x$ equals $2$, then $x$ is an isolated
point of $E$ by~\cite[Proposition 16]{BH59}. Since we have assumed $E$
to be connected and to contain more than one point, this does not occur.  

If the $C$-rank of $E$ equals $0$ at all points in a
neighbourhood of $x$, then by~\cite[Proposition 16]{BH59} $E\subset D$
is a real analytic submanifold of dimension $2$, and thus $E=D$. Since
we have assumed $h$ to be nonconstant, this does not occur.

Thus for each $x\in E$ the maximal $C$-rank of $E$ at nearby points
equals $1$. Now by~\cite[Proposition 16]{BH59} the set $E_1'\subset E$
of points in $E$ of $C$-rank $0$ is a proper $C$-analytic subset of $E$,
hence consists of finitely many points, and the set $E_0'=E\setminus
E_1'$ of points in $E$ of $C$-rank $1$ is a $1$-dimensional real
analytic submanifold of $D$, hence a finite union of embedded curves. 
Now we can prove parts (a-c) of the lemma. 

For part (a), it suffices to show $E_0\subset E_0'$. To see this,
consider $x\in E_0$, i.e.~$o(x)=k\geq 1$ as defined above. Let 
$\frac{\p^kh}{\p x_1^i\p x_2^j}(x)\neq 0$ be a nonvanishing partial
derivative of order $k$ at $x$. Suppose $i>0$ (the case $j>0$ is
analogous). Since $k$ is the lowest order of a nonvanishing partial
derivative of $h$ on $E$, the real analytic function
$g:=\frac{\p^{k-1}h}{\p x_1^{i-1}\p x_2^j}:D\to\R$ vanishes on $E$ and
thus belongs to the vanishing ideal of $E$ at $x$. Since $\frac{\p
  g}{\p x_1}(x)\neq 0$, it follows that the $C$-rank of $x$ equals $1$
and thus $x\in E_0'$. 

Next we prove (b). According to~\cite[Part I, Sections 1.4 and
  1.7]{GM88}, each $x\in 
E_1$ has an open neighbourhood $U\subset D$ such that $E\cap U$ is a
finite union $A_1\cup\cdots\cup A_m$ of $C^0$-embedded half-open
intervals that meet at the $x$ and are disjoint away from $x$. That
the $A_i$ are actually $C^1$-embedded is a consequence of the curve
selection lemma as stated in~\cite[Proposition 2.2]{AW05}. 

For part (c), consider $x\in E_0\subset E_0'$. Since the $C$-rank of
$E$ at $x$ equals $1$, the vanishing ideal of $E^*$ at $x$ is
generated by one holomorphic function $g^*$ with $d_xg^*\neq 0$. 
(To see this, pick local holomorphic coordinates $(z_1,z_2)$ near $x$
in which $g^*(z_1,z_2)=z_1$.) The
real or imaginary part of $g^*$ then defines a real analytic function
$g:D\to\R$ with $d_xg\neq 0$ and generating the vanishing ideal of $E$
at $x$. Thus $h=gh_1$ for a real analytic function $h_1$. 
If $h_1(x)=0$, then the $C$-analytic set $h_1^{-1}(0)$ is contained in
$E$. If $h_1^{-1}(0)$ agrees with $E$ near $x$ we have $h_1=gh_2$ 
for a real analytic function $h_2$. Continuing inductively, after
$\ell$ steps we have $h=g^\ell h_\ell$ for a real analytic function $h_\ell$.
Since $h$ vanishes to order $k$ at generic points of $E$, this process
terminates at $h=g^kh_h$ for a real analytic function $h_k$ that
doesn't vanish identically of $E$. Since $o(x)=k$, we have $h_k(x)\neq 0$. 
After possibly replacing $h$ by $-h$ we may assume that $h_k(x)>0$,
and by replacing $g$ by $g\sqrt[k]{h_k}$ we can then achieve $h=g^k$. 
Note that if $k$ is odd, then the level sets $\{h=t\}$ near $t=0$
correspond to the level sets $\{g=\sqrt[k]{t}\}$ and thus form a
foliation.  
If $k$ is even, then the level sets $\{h=t\}$ are empty for $t<0$
and correspond to the level sets $\{g=\pm\sqrt[k]{t}\}$ for $t\geq 0$,
so again they form a foliation. 
\end{proof}

\begin{remark}
As was shown in~\cite{EG00a} in the case of the $3$-sphere,
Lemma~\ref{lem:real-anal-sets} suffices to prove
Corollary~\ref{cor:real-anal-per}. For this, consider a stationary
solution $X$ of the Euler equations with real analytic Bernoulli
function $h$ and without zeroes or periodic orbits. We apply
Lemma~\ref{lem:real-anal-sets} to the restriction of $h$ to a real
analytic embedded disk $D\subset M$ transverse to $X$ and a connected
component $E$ of $S\cap D$, where $S$ is a connected component of a
level set $h^{-1}(c)$.  \\
If $h|_D$ were constant, then $h$ would be constant on $M$ by unique
continuation, so $X$ would be the Reeb vector field of a stable
Hamiltonian structure. Since $M$ is not a $T^2$-bundle over $S^1$, $X$
would have a periodic orbit by~\cite{HT09}, which we have excluded by
assumption. So $h|_D$ is nonconstant. If a connected component $E$
consisted of a single point, the corresponding set $S$ would be a
periodic orbit, so again this does not occur. Similarly, each point of
the singular set $E_1\subset E$ would give rise to a periodic orbit
(see the discussion preceding Lemma~\ref{lem:sing-orbit}),
so we must have $E_1=\emptyset$. \\
Now Lemma~\ref{lem:real-anal-sets} (c) shows that the
level sets of $h$ define a foliation of $M$ by invariant $2$-tori. We
claim that there exists an embedded closed curve $\gamma$ transverse
to the leaves of this foliation. To see this, pick any curve
$\beta:\R\to M$ parametrized by arclength and orthogonal to the leaves
(with respect to some Riemannian metric); then $h\circ\beta:\R\to[\min\,
  h,\max\,h]$ takes some value $c$ infinitely often, and since
$h^{-1}(c)$ is a union of finitely many $2$-tori, $\beta$ meets some
$2$-torus leaf of the foliation twice and we can close it up along
this leaf to obtain $\gamma:S^1\to M$ transverse to the
leaves. Since $M$ is connected, $\gamma$ meets every leaf and we can
choose it to intersect each leaf exactly once. Assigning to $x\in M$
the value $t\in S^1$ for which $\gamma(t)$ lies on the leaf through
$x$ defines a $2$-torus bundle $T^2\to M\to S^1$, contradicting the
hypothesis on $M$.  
\end{remark}

Let us now return to a solution $(X,\lambda,h)$ of the Euler
equations 
\begin{equation}\label{eq:Euler-stab}
   i_Xd\lambda=-dh,\qquad L_X\mu=0,\qquad \lambda(X)>0. 
\end{equation}
We first observe that the simplest obstruction to geodesibility in
the sense of~\cite{Su78} vanishes: 

\begin{lemma}\label{lem:Reeb-comp}
The foliation by flow lines of $X$ has no Reeb components. 
\end{lemma}

\begin{proof}
Suppose that $Z\subset M$ is a Reeb component, i.e., an embedded
cylinder invariant under the flow such that the flow direction induces
the boundary orientation on both boundary circles, and all flow lines
that meet the interior of $Z$ converge to the boundary in forward and
backward time. Since $h$ is invariant under the flow of $X$, the
last property implies that $h$ is constant on $Z$. Thus $i_Xd\lambda =
-dh$ vanishes on $Z$ and we get the contradiction
$$
   0 = \int_Zd\lambda = \int_{\p Z}\lambda > 0. 
$$
\end{proof}

Suppose now that the Bernoulli function $h:M\to\R$ is real analytic. 
Consider a connected component $S$ of a singular level set
$h^{-1}(c)$. After replacing $h$ by $h-c$ we may assume that
$c=0$. For $x\in S$ we denote by $o(x)\geq 1$ the vanishing order of
$h$ at $x$, i.e., the lowest order of a nonvanishing derivative of $h$
at $x$. Set $k:=\min\{o(x)\mid x\in S\}$. Since $h$ is invariant under
$X$, so are the sets 
$$
   S_1:=\{x\in S\mid o(x)>k\},\qquad S_0:=S\setminus S_1.
$$ 

\begin{lemma}\label{lem:sing-orbit}
(a) $S_1$ is a finite union of periodic orbits, and $S_0$ is a
(possibly empty, possibly noncompact, possibly disconnected) embedded
surface.  

(b) Let $U\subset M$ be an open tubular neighbourhood of a connected
component $\gamma$ of $S_1$, and suppose that $S_0\cap
U\neq\emptyset$. Then $U\cap S$ is a finite union $Z_1\cup\cdots\cup
Z_n$ of $C^1$-immersed invariant half-open cylinders whose interiors
are embedded in $U\setminus\gamma$ and whose boundaries are $d$-fold
coverings of $\gamma$, for some $d\in\N$. 
\end{lemma}

\begin{proof}
We apply Lemma~\ref{lem:real-anal-sets} to the restriction of $h$ to
an embedded disk $D\subset M$ transverse to $X$ and a connected
component $E$ of $S\cap D$. Then $E_i=S_i\cap D$, $i=0,1$, and part
(a) follows directly from Lemma~\ref{lem:real-anal-sets} (a). 

For part (b), we choose the disk $D$ centered at $\gamma$. By
Lemma~\ref{lem:real-anal-sets} (b), $E\cap U$ is a finite union
$A_1\cup\cdots\cup A_m$ of $C^1$-embedded arcs 
that meet at the origin and are disjoint outside the origin. 
The return map $\phi:D\to D$ of the flow of $X$ acts as a permutation
on the arcs, and the orbits of this action give rise to the surfaces
$Z_1,\dots,Z_n$ (for some $n\leq m$).  
We number the arcs in counterclockwise order, so $\phi$ induces a
permutation $\bar\phi$ of the set $\{1,\dots,m\}$. Since $\phi$ is a
homeomorphism, it must map adjacent arcs to adjacent arcs, so the
permutation $\bar\phi$ satisfies $\bar\phi(k+1)=\bar\phi(k)\pm 1$ for
all $k$. In other words, $\bar\phi$ gives an element in the orthogonal
group $O(2)$ acting on the $m$-th roots of unity. Since $\phi$
preserves the area form induced by $\om=i_X\mu$ on $D$, it is orientation
preserving and $\bar\phi\in SO(2)$ is a rotation. Write $m=nd$, where
$n$ is the number of orbits of $\bar\phi$ and $d$ the number of points
in one orbit. Then each orbit corresponds to a $C^1$-immersed
half-open cylinder whose interior is embedded and whose boundary is
the $d$-fold covering of $\gamma$. 
\end{proof}

\begin{remark}
For $d=2$ in Lemma~\ref{lem:sing-orbit} each $Z_i$ is homeomorphic to
an embedded open M\"obius strip. This situation arises for example at
negative hyperbolic orbits in Reeb flows, see e.g.~\cite{HT09}. 
\end{remark}

According to the Arnold-Liouville theorem~\cite{Arn66}, each regular
level set of $h$ is a disjoint 
union of finitely many embedded $2$-tori on which the flow of $X$ is
linear. The following proposition
gives a similarly precise description of the flow on the singular
level sets in the case that the Bernoulli function $h:M\to\R$ is real
analytic. 

\begin{prop}\label{prop:sing-level}
Let $(X,\lambda,h)$ be a solution of~\eqref{eq:Euler-stab} on a closed
oriented $3$-manifold $M$ with $h$ real analytic. 
  
(a) The function $h:M\to\R$ has finitely many critical values. 

(b) Each singular level set of $h$ is a finite disjoint union of
embedded $X$-invariant sets that are periodic orbits, $2$-tori, Klein
bottles, open cylinders, or open M\"obius strips. 
The closures of the open cylinders and M\"obius strips are
$C^1$-immersed closed cylinders and M\"obius strips whose boundary
components are $d$-fold coverings of periodic orbits, where finitely
many of them can meet at the same periodic orbit. 

(c) On each $2$-torus as in (b) the flow is linear (rational or
irrational). On each Klein bottle the flow is linear and periodic.   
On each M\"obius strip the flow is periodic. 
On each cylinder the flow is either periodic, or all flow lines 
converge to the boundary in forward and backward time. 

(d) On each cylinder as in (b) the flow direction induces the boundary
orientation on one boundary circle, and the opposite orientation on the
other one.  
\end{prop}

\begin{proof}
(a) The critical point set $C$ of $h$ is a compact analytic subset
of $M$, hence its image $h(C)$ is a compact semi-analytic subset of
$\R$. Since every compact semi-analytic subset of $\R$ is a finite
union of points and closed intervals, and $h(C)$ contains no
intervals by Sard's theorem, the set $h(C)$ is finite.   

(b) Consider a connected component $S$ of a singular level set
$h^{-1}(c)$ and define $S_0$, $S_1$ as above. By
Lemma~\ref{lem:sing-orbit} (a), $S_1$ is a finite union of periodic orbits
and each connected component $\mathring{Z}$ of $S_0$ is an embedded surface.  
Since $X$ is tangent to $\mathring{Z}$, the Euler characteristic of $\mathring{Z}$ vanishes,
so $\mathring{Z}$ can only be a torus, Klein bottle, open cylinder, or open
M\"obius strip. By Lemma~\ref{lem:sing-orbit} (b), the closure $Z$ of $\mathring{Z}$ 
is a $C^1$-immersed closed cylinder resp.~M\"obius strip whose boundary
components belong to are $d$-fold coverings of periodic orbits in
$S_1$. Since only finitely many cylinders and M\"obius strips meet at
each orbit in $S_1$, the number of components of $S_0$ is finite. 

(c) We continue in the notation from (b). 
Lemma~\ref{lem:real-anal-sets} (c) shows that near each
$x\in S_0$ (after possibly replacing $h$ by $-h$ if $k$ is even) we
can write $h=g^k$ for a real analytic function $g$ near $x$. 
Now we distinguish 2 cases. 

{\bf Case 1: }$k$ is odd. 

Then $S_0$ is a regular level set of
the real analytic $X$-invariant function $g=\sqrt[k]h$ defined on a
neighbourhood of $S_0$. In particular, $S_0$ is orientable, so the
closure $Z$ of each connected component is either a torus or a
cylinder. By the Arnold-Liouville theorem~\cite{Arn66}, on each torus the flow 
is linear. For a cylinder, the flow preserves the smooth area form
$\mu/dg$ on $\mathring{Z}$ (where $\mu$ is the invariant volume form on $M$,
and the total area may be infinite). Thus the return map on each local
transverse slice preserves a smooth measure. This implies that the
flow is either periodic, or all flow lines on $\mathring{Z}$ converge to $\p
Z$ in forward and backward time. 

{\bf Case 2: }$k$ is even.

Consider again the closure $Z$ of a
connected component of $S_0$. If $Z$ is orientable, then so is its
normal bundle in $M$, hence we can choose a $k$-th root
$g=\sqrt[k]{h}$ near $\mathring{Z}$ (requiring it to be positive on one side
of $Z$) and proceed as in Case 1. If $Z$ is
non-orientable, consider its orientable 2-1 covering $\wt Z\to Z$. 
Pull back the normal bundle to $Z$ to obtain an orientable 2-1 covering
$\wt U\to U$ of a neighbourhood $U$ of $Z$ in $M$. In particular, the
normal bundle of $\wt Z$ in $\wt U$ is trivial, so we can choose a
$k$-th root $\wt g=\sqrt[k]{\wt h}$ of the pullback $\wt h$ of $h$
near the interior of $\wt Z$ and proceed as in Case 1 on the covering. If $Z$ is a
Klein bottle, then $\wt Z$ is a torus on which the flow is
linear. Since the data on $\wt Z$ were invariant under the covering
involution, the linearizing coordinates on $\wt Z$ can be chosen
to descend to coordinates on $Z$ in which the flow is linear. It is
easy to see that every linear foliation of the Klein bottle is periodic. 
If $Z$ is a M\"obius strip, then $\wt Z$ is
a cylinder on which the flow is either periodic, or all flow lines
converge to $\p\wt Z$ in forward and backward time. Now the foliation on
$\wt Z$ defined by the flow is invariant under the covering
involution, which in suitable coordinates on $\wt Z\cong
\R/\Z\times[-1,1]$ is given by $(x,y)\mapsto(x+1/2,-y)$. 
This excludes 
convergence of flow lines to $\p\wt Z$ in forward and backward time,
so the flow on $\wt Z$ is periodic and descends to a periodic flow on
$Z$. 

(d) follows directly from Lemma~\ref{lem:Reeb-comp}. 
\end{proof}

\begin{example}\label{ex:Klein-bottle}
On $\R^3$ with coordinates $(x,y,z)$ consider a solution of~\eqref{eq:Euler-stab} given by
$$
   X=\p_x,\quad \lambda=h(z)dz,\quad \om=dx\wedge dy,\quad \mu=dx\wedge dy\wedge dz
$$
with a positive Bernoulli function $h(z)$. As noted in the
Introduction, this yields a solution of $\nabla_XX=-\nabla p$ with
$p(z)=h'(z)/2$ for any metric satisfying $i_Xg=\lambda$, i.e., 
$$
   |\p_x|^2=h(z),\quad \la\p_x,\p_y\ra = \la\p_x,\p_z\ra = 0. 
$$
Note that ${\rm curl}\,X=h'(z)\p_y$. 
An example of such a metric with ${\rm vol}_g=dx\wedge dy\wedge dz$ is
$$
   g = h(z)dx^2 + dy^2 + h(z)^{-1}dz^2. 
$$
Consider now the closed oriented $3$-manifold $M:=\R^3/\sim$ obtained by dividing out the 
equivalence relation generated by
$$
   (x,y,z)\sim (x,y+1,z)\sim (x,y,z+1)\sim (x+1,-y,-z). 
$$
It is doubly covered by the $3$-torus via the map
$$
   T^3=\R^3/\Z^3\to M,\qquad [x,y,z]\mapsto[2x,y,z]. 
$$
One can also view $M$ as the mapping torus of the map
$$
   T^2\to T^2,\qquad (y,z)\mapsto(-y,-z). 
$$
The data $X,\lambda,\om,\mu,g$ above descend to $M$ provided the function $h$ satisfies
$$
   h(z) = h(-z) = h(z+1). 
$$
An example of such a function is $h(z) = 2-\cos(2\pi z)$. The level sets $\{z\in\Z\}$ and $\{z\in\frac{1}{2}+\Z\}$
are Klein bottles, and all other level sets are $2$-tori on which the flow of $X$ is linear and periodic. This can be seen
from the map 
$$
   T^2\to \{z\in\pm c+\Z\}\subset M,\qquad [x,y]\mapsto[2x,y,c]
$$ 
which is a 2-1 covering if $c\in\Z$ or $z\in\frac{1}{2}+\Z$, and a diffeomorphism otherwise. 
One can generalize this example to mapping tori of the shear maps 
$$
   (y,z)\mapsto(y+\ell z,z) \text{ or }(y,z)\mapsto(-y+\ell z,-z),\qquad \ell\in\Z. 
$$
This example shows that Klein bottles (and thus also M\"obius
strips) can actually occur in singular level sets. 
\end{example}

Now we construct a stabilizing $1$-form $\nu$ near the singular level
sets. 
We assign a {\em covering number} $d_\gamma$ to each simple periodic
orbit $\gamma$ on a singular level set as follows. If $\gamma\subset
S_0$ has nontrivial normal bundle in $S_0$ we let $d_\gamma:=2$ (this
occurs for the central orbit on a M\"obius strip and the two special
orbits on a Klein bottle). If $\gamma\subset S_1$ is $d$-fold covered
by the boundary of a component in $S_0$ (as in
Lemma~\ref{lem:sing-orbit}) we set $d_\gamma:=d$. In all other cases
we set $d_\gamma:=1$. 

\begin{prop}\label{prop:nu}
There exist a closed $1$-form $\nu$ on a neighbourhood of the
union of the singular level sets of $h$ satisfying $\nu(X)>0$, and
normalized such that $\int_\gamma\nu=1/d_\gamma$ for each simple
periodic orbit $\gamma$ on a singular level set. 
\end{prop}

\begin{remark}\label{rem:nu}
We still have the freedom to multiply $\nu$ by a positive constant on
each component of a singular level set that is either an 
isolated periodic orbit or a torus with irrational flow. 
\end{remark}

\begin{proof}
In view of Proposition~\ref{prop:sing-level} (a), it suffices to consider one
connected component $S$ of a singular level set $h^{-1}(c)$ at a time.
If $S$ consists only of an isolated periodic orbit $\gamma$ we pick a
closed $1$-form $\nu$ near $\gamma$ with $\nu(X)>0$ and
$\int_\gamma\nu=1$. 

If $S$ is a torus or a Klein bottle, then the flow
defines a linear foliation on $S$. Pick a transverse linear foliation
$\FF$ and define a closed $1$-form $\nu$ on $S$  by $\nu(X):=a$ and
$\nu|_{T\FF}=0$, for some constant $a>0$. Then extend $\nu$ to a
closed $1$-form on a neighbourhood of $S$ in $M$. If the flow is
periodic on $S$ with minimal period $T$ we choose $a=1/T$ to satisfy
the normalization condition $\int_\gamma\nu=1/d_\gamma$. On a cylinder
with irrational flow we can choose $a$ arbitrarily (this freedom will
still be needed later). 

It remains to consider the case that $S$ is a union of cylinders and
M\"obius strips whose boundaries meet in a collection $S_1$ of
periodic orbits. We first pick a closed $1$-form $\nu_1$ on a
neighbourhood of $S_1$ in $M$ with $\nu_1(X)>0$. Using
Lemma~\ref{lem:sing-orbit}, we can normalize $\nu_1$ on each periodic
orbit $\gamma$ in $S_1$ by $\int_\gamma\nu_1=1/d_\gamma$.
Next we consider a cylinder or M\"obius strip $Z$ in
$S$. The chosen form $\nu_1$ provides a closed $1$-form near $\p Z$
with $\nu_1(X)>0$. Our normalization implies that the integral of
$\nu_1$ over each boundary component of $Z$ (oriented by $X$) equals
$1$. 

{\bf Claim: }There exists an embedded arc $\delta$ in $Z$ which starts and
ends on $\p Z$ and is transverse to the foliation by $X$-orbits. 

To see this, consider first a cylinder $Z$ on which the flow is
periodic. Then we find a submersion $\tau:Z\to[-1,1]$ (the projection
onto the space of leaves) having the $X$-orbits as regular level sets, 
and any curve $\delta:[-1,1]\to Z$ with $\tau\circ\delta(t)=t$ has
the desired properties. 

Next consider a M\"obius strip $Z$. By Proposition~\ref{prop:sing-level} (b),
the flow on $Z$ is periodic and thus defines a foliation with closed
leaves. Leaves of this foliation can only represent the homology
classes $c$ or $2c$, where $c$ is the generator of $H_1(Z;\Z)$ such
that $\p Z$ oriented by $X$ represents the class $2c$. Now embedded
curves in class $c$ have nontrivial normal bundle, while embedded
curves in class $2c$ have trivial normal bundle. It follows that not
all leaves can lie in class $2c$, since otherwise the foliation would
be cooriented and thus foliate a cylinder rather than a M\"obius
strip. On the other hand, since any two closed curves in class $c$
intersect, there can be at most one leaf in class $c$. Hence there is
exactly one leaf $\gamma$ in class $c$. Pick an arc $\delta_0$
intersecting transversely a closed invariant neighbourhood $A$ of
$\gamma$. Since $Z\setminus\mathring{A}$ is a cylinder foliated by closed
leaves, we can connect the endpoints of $\delta_0$ to $\p Z$ by two
disjoint arcs transverse to the foliation to obtain the desired arc
$\delta$. 

Finally, consider a cylinder $Z$ on which each orbit converges to $\p
Z$ in forward and backward time. Denote by $\p_+Z$ (resp.~$\p_-Z$) the
boundary component whose boundary orientation is induced by $X$
(resp.~$-X$). This is possible in view of Proposition~\ref{prop:sing-level}
(d). Pick a short arc $\delta_1$ starting on $\p_-Z$ and transverse to
the foliation. Recall that $-X$ defines the boundary orientation on
$\p_-Z$, which means that $(\eta,-X)$ is a positive basis for an
outward pointing vector $\eta$ at $\p_-Z$. Since $\dot\delta_1$ is inward
pointing, $(\dot\delta_1,X)$ is a positive basis along $\delta_1$. We
extend $\delta_1$ to an arc $\delta_2$ ending close to $\p_+Z$ by
moving almost parallel to the backward $X$-orbit starting from the endpoint of
$\delta_1$, at a slight angle to make $\delta_2$ transverse to the
foliation. It follows that $(\dot\delta_2,X)$ is a positive basis at
the endpoint of $\delta_2$. Since $(\eta,X)$ is also a positive basis
for an outward pointing vector $\eta$ at $\p_+Z$, we can connect the
endpoint of $\delta_2$ to $\p_+Z$ to obtain the desired arc
$\delta$. This proves the claim. 
\smallskip

We can choose the arc $\delta$ in the claim such that
$\nu_1(\dot\delta)=0$ near $\p Z$. By abuse of notation, we will
denote the image of $\delta$ in $Z$ again by the letter $\delta$. For
$x\in Z\setminus\delta$, we denote by $\gamma_x$ the $X$-orbit through
$x$ starting and ending on $\delta$ (oriented by $X$). Since $\nu_1$
is closed, normalized by $\int_\gamma\nu_1=1$ for each boundary orbit
$\gamma$ (oriented by $X$), and $\nu_1(\dot\delta)=0$, it follows from
Stokes' theorem that $\int_{\gamma_x}\nu_1=1$ for all $x$ sufficiently
close to $\p Z$. 

We extend $\nu_1$ to a (not necessarily closed) $1$-form $\nu_2$ on
$Z$ satisfying
\begin{enumerate}
\item $\nu_2(X)>0$,
\item $\nu_2=\nu_1$ near $\p Z$,
\item $\int_{\gamma_x}\nu_2=1$ for all $x\in Z\setminus\delta$.  
\end{enumerate}
We define a function $\phi:Z\to\R/\Z$ by $\phi|_\delta:=0$ and 
$
   \phi(x) := \int_{\gamma_x^-}\nu_2
$
for $x\notin\delta$, where $\gamma_x^-$ is the part of $\gamma_x$
starting on $\delta$ and 
ending at $x$. Property (iii) shows that $\phi$ is well-defined and
smooth as a map to $\R/\Z$. We claim that the closed $1$-form
$\nu:=d\phi$ on $Z$ agrees with $\nu_1$ near $\p Z$ and satisfies
$\nu(X)>0$. 

For the first statement, recall that $\nu_1$ is closed on a tubular
neighbourhood $U$ of $\p Z$. In view of (ii), this implies that
$\phi(x)=\int_\gamma\nu_1\in\R/\Z$ for $x\in U$ and $\gamma$ any path
in $U$ from $\delta$ to $x$. 
A standard calculation shows that
$\nu=d\phi=\nu_1$ on $U$. The second statement follows directly from
(i) and the definition of $\phi$, which implies that $\nu(X) =
d\phi(X) = \nu_2(X) > 0$. 

To conclude the proof of Proposition~\ref{prop:nu}, we extend $\nu$ to a
closed form on a neighbourhood of $Z$ in $M$ which agrees with $\nu_1$
near $\p Z$. Performing this for all cylinders and M\"obius strips in
$S$, we obtain an extension of $\nu_1$ to a closed $1$-form $\nu$ on
a neighbourhood of $S$ with $\nu(X)>0$. By construction $\nu$
satisfies the normalization condition $\int_\gamma\nu=1/d_\gamma$. 
\end{proof}

The proof of Proposition~\ref{prop:nu} yields

\begin{cor}\label{cor:nonconst}
If in Theorem~\ref{thm:real-anal-stable} the Bernoulli function is
nonconstant, then $M$ is 
the union of finitely many connected Seifert manifolds glued along
their torus boundary components. 
\end{cor}

\begin{proof}
The proof of Lemma~\ref{lem:sing-orbit} shows that each component
$\gamma$ of $S_1$ in a component $S=S_0\cup S_1$ of a singular
level set of $h$ has a tubular neighbourhood $U_\gamma$ such that
$S\cap U_\gamma$ consists of $n$ $C^1$-immersed half-open cylinders
whose boundaries are $d_\gamma$-fold coverings of $\gamma$. Moreover, the
return map on a disk $D$ transverse to $\gamma$ acts as a rotation
$i\mapsto i+p$ on the set of $m=nd_\gamma$ arcs formed by $D\cap S$,
where $gcd(p,m)=n$. It follows that $U_\gamma$ is the mapping
torus of a rotation of $D$ by an angle $2\pi p/m$, 
thus $U_\gamma$ is Seifert fibered with a circle action going
$d_\gamma$ times along the mapping torus and with exceptional orbit
$\gamma$.  

If $Z$ is the closure of a cylinder or M\"obius strip in $S_0$, then the circle
action near its boundary can be extended to a circle action on a
neighbourhood of $Z$ in $M$ preserving $Z$. Tori or Klein
bottles in $S_0$ also have neighbourhoods with circle actions
preserving them. Altogether, we find a circle action on a closed
neighbourhood $U$ of the union $N$ of the singular level sets, which is
free except for the orbits with covering number $d_\gamma>1$ arising
in $S_1$ or as exceptional orbits on M\"obius strips and Klein
bottles. Thus each component of $U$ is a Seifert manifold whose
boundary is a union of $2$-tori. Since all these $2$-tori are pairwise
connected by integrable regions $(a,d)\times T^2$ (see
Section~\ref{sec:reg}), we obtain the manifold $M$ by gluing together
these pairs of $2$-tori.  
\end{proof}

\begin{remark}\label{rem:nonconst}
(i) The gluing of the boundary $2$-tori need not match the directions
of the circle actions, so the manifold $M$ itself need not be Seifert
fibered. In $3$-manifold terminology, if $M$ is irreducible,
then its JSJ decomposition contains no atoroidal pieces, and its
decomposition according to the geometrization conjecture contains only
pieces with spherical geometry. \\
(ii) If the singular level sets of $h$ consist only of $2$-tori, then
$M$ is the mapping torus of a shear map $(y,z)\mapsto(y+\ell z)$ on
$T^2$ for some $\ell\in\Z$. \\
(iii) If in Theorem~\ref{thm:real-anal-stable} the Bernoulli function
is constant, then $d\lambda=f\om$ for a function $f:M\to\R$ which
is again invariant under the flow of $X$. If $f$ is real analytic and
nonconstant, then again the conclusions of
Corollary~\ref{cor:nonconst} hold. If $f$ equals a constant
$c$, then $X$ is either the Reeb vector field of a (positive or
negative) contact structure (if $c\neq 0$), or the horizontal vector
field of a mapping torus (if $c=0$). See~\cite{CV10} for further
discussion. 
\end{remark}

\section{Extension of $\nu$ over the regular level sets}\label{sec:reg}

According to Proposition~\ref{prop:nu},
there exist a closed $1$-form $\nu$ on an open neighbourhood $U$ of
the union $N$ of the singular level sets of $h$ satisfying $\nu(X)>0$, 
and normalized such that $\int_\gamma\nu=1/d_\gamma$ for each simple
periodic orbit $\gamma$ on a singular level set. 
We fix such $\nu$ for the remainder of this section. 
Consider a connected component $V$ of $M\setminus N$. By the
Arnold-Liouville theorem (\cite{Arn66}, see also~\cite{AK98}), 
$V$ is diffeomorphic to $(a,d)\times T^2$
such that $h(r,x)=r$ and the vector field $X$ is constant on each
torus $\{r\}\times T^2$. The following lemma allows us to interpolate
between two stabilizing closed $1$-forms on such an integrable region
$V$. 

\begin{lemma}\label{lem:interpol}
Let $\om$ be a nowhere vanishing closed $2$-form on $[a,b]\times T^2$ with vector
field $X$ generating $\ker\om$ and tangent to the tori $\{r\}\times
T^2$. Let $\nu_0,\nu_1$ be two closed $1$-forms with
$\nu_i(X)>0$ and $[\nu_0]=[\nu_1]\in H^1([a,b]\times
T^2;\R)$. Then there exists a closed $1$-form $\wt\nu$ with
$\wt\nu(X)>0$ which agrees with $\nu_0$ near $\{a\}\times T^2$
and with $\nu_1$ near $\{b\}\times T^2$.  
\end{lemma}

\begin{proof}
Since $[\nu_0]=[\nu_1]$, we can write $\nu_1=\nu_0+dg$
for some function $g$. Pick a cutoff function $\phi:[a,b]\to[0,1]$
which equals $0$ near $a$ and $1$ near $b$. We denote by $r$ the
coordinate on $[a,b]$, viewed as a function on $[a,b]\times T^2$. 
Then the $1$-form
$$
   \wt\nu := \nu_0+d\bigl(\phi(r)g\bigr)
$$
is closed and agrees with $\nu_0$ near $\{a\}\times T^2$,
and with $\nu_1$ near $\{b\}\times T^2$. Using
$dg=\nu_1-\nu_0$, we can rewrite $\wt\nu$ as
$$
   \wt\nu = \nu_0 + \phi(r)dg + g\phi'(r)dr
   = \bigl(1-\phi(r)\bigr)\nu_0 + \phi(r)\nu_1 + g\phi'(r)dr.   
$$
Since $\nu_i(X)>0$ and $dr(X)=0$, we see that 
$$
   \wt\nu(X) = \bigl(1-\phi(r)\bigr)\nu_0(X) +
   \phi(r)\nu_1(X) > 0, 
$$
so $\wt\nu$ has the desired properties.  
\end{proof}

Pick values $a<b<c<d$ such that the form
$\nu$ is defined near 
$$
   V_0=\bigl((a,b]\cup[c,d)\bigr)\times T^2. 
$$
Let $\bar\nu$ be the $T^2$-invariant $1$-form on $V_0$ obtained by
averaging $\nu$ over the tori $\{r\}\times T^2$. It is closed,
represents the same cohomology class as $\nu$, and still satisfies
$\bar\nu(X)>0$ (due to $T^2$-invariance of $X$). By
Lemma~\ref{lem:interpol}, we find a closed $1$-form $\wt\nu$ on $V_0$
with $\wt\nu(X)>0$ which agrees with $\nu$ near $\{a,d\}\times T^2$
and with $\bar\nu$ near $\{b,c\}\times T^2$.   

Suppose first that the direction of the vector field $X$ is {\em not}
constant in $r\in(a,d)$. By choosing $b,c$ sufficiently close to $a,d$
we can arrange that  the direction of $X$ is not
constant in $r\in(b,c)$. Then Proposition 3.14 in~\cite{CV10} provides an
extension of $\wt\nu$ to a stabilizing $1$-form over $V$.

Next suppose that the direction of the vector field $X$ is 
constant in $r\in(a,d)$. In this case, there exists an obstruction to
extending $\wt\nu$ to a stabilizing $1$-form over $V$. To describe it, 
we write 
$$
   X=\rho(r)\bar X
$$ 
for a constant vector field $\bar X$ on $T^2=\R^2/\Z^2$. For each
$r\in(a,d)$ we define a $1$-current on $V$ by 
\begin{equation}\label{eq:current}
   c_r(\alpha) := \int_{\{r\}\times T^2}\alpha(\bar X)d\theta\,d\phi
\end{equation}
for a $1$-form $\alpha$. Thus $c_r(\alpha)$ is the pairing of the
average of $\alpha$ over $\{r\}\times T^2$ with the vector $\bar X$. 
It is shown in Lemma~\ref{lem:current} below that $c_r(\nu)$ equals a
constant $c_-(\nu)$ for $r\in(a,b)$ and a constant $c_+(\nu)$ for
$r\in(c,d)$, and $c_-(\nu)=c_+(\nu)$ is a necessary condition for
extending $\wt\nu$ to a stabilizing $1$-form over $V$. On the other
hand, this condition is also sufficient according to Lemma 3.10
in~\cite{CV10}. Performing this construction for all components $V$, 
Theorem~\ref{thm:real-anal-stable} follows once we can show that this
condition can always be satisfied, which is the content of


\begin{prop}\label{prop:fol-cycle}
We can choose the $1$-form $\nu$ in Proposition~\ref{prop:nu} such that
$c_-(\nu)=c_+(\nu)$ for all integrable regions $V\subset M\setminus N$
on which the direction of $X$ is constant. 
\end{prop}

The proof of this proposition is based on an analysis of the currents
$c_r$. Throughout the following discussion, we always need to
distinguish the following two cases.

{\bf Case 1: }The direction of $\bar X$ is rational. 

In this case, we normalize $\bar X$ and choose coordinates
$(\theta,\phi)$ on $T^2=\R^2/\Z^2$ such that $\bar X=\p_\phi$, so the
periodic orbits of $\bar X$ are the circles $\{(r,\theta)\}\times S^1$
and have period $1$. It follows that
\begin{equation}\label{eq:c-rat}
   c_r(\alpha) =
   \int_0^1\Bigl(\int_{\{(r,\theta)\}\times S^1}\alpha\Bigr)d\theta. 
\end{equation}

{\bf Case 2: }The direction of $\bar X$ is irrational. 

In this case, since an irrational linear flow on the torus is ergodic, 
as a direct consequence of Birkhoff's ergodic theorem we have 
\begin{equation}\label{eq:c-irrat}
   c_r(\alpha) =
   \lim_{T\to\infty}\frac{1}{T}\int_0^T\alpha\bigl(\dot\gamma(t)\bigr)dt
\end{equation}
for each orbit $\gamma$ of $\bar X$ on $\{r\}\times T^2$. 

\begin{lemma}\label{lem:current}
In both cases, the currents $c_r$ have the following properties. 

(a) $c_r$ is closed and invariant under $\bar X$. 

(b) The homology class $[c_r]\in H_1(T^2;\R)$ is independent of $r$. 

(c) $c_r(\wt\nu)$ is independent of $r$ for each $1$-form $\wt\nu$
   satisfying $i_Xd\wt\nu=0$. 

(d) For all $r_1,r_2$ and $\lambda$ satisfying $i_Xd\lambda=-dh$ ($=-dr$),
$$
   c_{r_2}(\lambda)-c_{r_1}(\lambda) =
   \int_{r_1}^{r_2}\frac{1}{\rho(r)}dr.
$$ 
\end{lemma}

\begin{proof}
We prove everything in Case 2; the proof in Case 1 is similar but
easier, using formula~\eqref{eq:c-rat} instead of~\eqref{eq:c-irrat}.  

(a) For closedness, note that for each function $f:V\to\R$ we have
\begin{align*}
   \p c_r(f) = c_r(df) 
   = \lim_{T\to\infty}\frac{1}{T}\int_0^Tdf\bigl(\dot\gamma(t)\bigr)dt
   = \lim_{T\to\infty}\frac{1}{T}\Bigl(f\bigl(\gamma(T)\bigr) -
   f\bigl(\gamma(0)\bigr)\Bigr)dt = 0
\end{align*}
because the term $f\bigl(\gamma(t)\bigr) - f\bigl(\gamma(0)\bigr)$
is bounded uniformly in $T$. For invariance, let $\phi_\tau$ be the
flow of $\bar X$. Then we get
$$
   c_r(\phi_\tau^*\alpha) 
   = \lim_{T\to\infty}\frac{1}{T}\int_{\gamma([0,T])}\phi_\tau^*\alpha
   = \lim_{T\to\infty}\frac{1}{T}\int_{\gamma([\tau,T+\tau])}\alpha
   = c_r(\alpha). 
$$
(b) For $r_1<r_2$ we define the $2$-current
$$
   C(\beta) := \lim_{T\to\infty}\frac{1}{T}\int_{r_1}^{r_2}\int_0^T
   \beta\bigl(\p_r,\dot\gamma(t)\bigr)dr\,dt 
   = \int_{r_1}^{r_2}c_r(i_{\p_r}\beta)dr. 
$$
Then we get
\begin{align*}
   \p C(\alpha) 
   &=
   \lim_{T\to\infty}\frac{1}{T}\int_{[r_1,r_2]\times\gamma([0,T])}d\alpha
   \cr
   &= \lim_{T\to\infty}\frac{1}{T}\Bigl\{
   \int_{[r_1,r_2]\times\gamma(0)}\alpha 
   + \int_{\{r_2\}\times\gamma([0,T])} \alpha
   - \int_{[r_1,r_2]\times\gamma(T)}\alpha 
   - \int_{\{r_1\}\times\gamma([0,T])} \alpha \Bigr\} \cr
   &= c_{r_2}(\alpha) - c_{r_1}(\alpha)
\end{align*}
because the first and third term in the big bracket are bounded
uniformly in $T$. 

(c) From the proof of (b) and $i_Xd\wt\nu=0$ we get
$$
   c_{r_2}(\wt\nu) - c_{r_1}(\wt\nu) 
   = \lim_{T\to\infty}\frac{1}{T}\int_{r_1}^{r_2}\int_0^T
   d\wt\nu\bigl(\p_r,\dot\gamma(t)\bigr)dr\,dt
   = 0
$$
because the integrand vanishes identically. 

(d) From $i_Xd\lambda=-dr$ and $\dot\gamma=\bar X=\frac{1}{\rho}X$ we get 
$$
   d\lambda\bigl(\p_r,\dot\gamma(t)\bigr)
   = -\frac{1}{\rho(r)}(i_Xd\lambda)(\p_r) = \frac{1}{\rho(r)}dr(\p_r) =
   \frac{1}{\rho(r)},
$$
and thus from the proof of (b),  
$$
   c_{r_2}(\lambda) - c_{r_1}(\lambda) 
   = \lim_{T\to\infty}\frac{1}{T}\int_{r_1}^{r_2}\int_0^T
   d\lambda\bigl(\p_r,\dot\gamma(t)\bigr)dr\,dt
   = \int_{r_1}^{r_2}\frac{1}{\rho(r)}dr.    
$$
\end{proof}

Property (a) means that $c_r$ it is a foliation cycle in the sense
of~\cite{Su76} supported on $\{r\}\times T^2$. 
For a smooth function $f:V\to\R$ we denote its average over
$\{r\}\times T^2$ by 
$$
   \ol{f}(r) := \int_{\{r\}\times T^2}f\,d\theta\,d\phi. 
$$
Let $\nu$ be the closed $1$-form on $V_0\subset V$ with $\nu(X)>0$ from
above and define $c_-=c_-(\nu):=c_r(\nu)$ for $r\in(a,b)$. Then for $r\in(a,b)$, 
$$
   0<c_-\equiv c_r(\nu) = \ol{\nu(\bar X)}(r) =
   \frac{1}{\rho(r)}\ol{\nu(X)}(r).  
$$
Since $\nu(X)$ is bounded away from $0$ and $\infty$, so is
$\ol{\nu(X)}(r)$, and hence by the preceding line also $\rho(r)$. Since 
$$
   c_r(\lambda) = \ol{\lambda(\bar X)}(r) =
   \frac{1}{\rho(r)}\ol{\lambda(X)}(r)
$$
and $\lambda(X)$ is bounded away from $0$ and $\infty$, so are
$\ol{\lambda(X)}(r)$ and $c_r(\lambda)$.

\begin{lemma}\label{lem:ends}
Let $V\subset M\setminus N$ be an integrable region on which the
direction of $\bar X$ is constant and let $S=\ol{V}\cap h^{-1}(a)$ or
$S=\ol{V}\cap h^{-1}(d)$. 

(a) If the direction of $\bar{X}$ is rational, then $S$ is one of the
following: 
\begin{enumerate}
\item an isolated periodic orbit,
\item a rational torus or a Klein bottle,
\item a union of periodic cylinders and M\"obius strips connected at
  their (possibly multiply covered) boundaries.
\end{enumerate}
(b) If the direction of $\bar{X}$ is irrational, then $S$ is one of the
following: 
\begin{enumerate}
\item an isolated periodic orbit,
\item an irrational torus.
\end{enumerate}
\end{lemma}

\begin{proof}
We consider the case $S=\ol{V}\cap h^{-1}(a)$, the other one being
analogous. 

In case (a), the flow of $X$ is periodic on each $\{r\}\times
T^2$ for $r>a$, with uniformly bounded period. Thus for each $p\in S$
we find a sequence of periodic orbits
$\gamma_n:[0,T_n]\to\{r_n\}\times T^2$ with $\gamma_n(0)\to p$ as
$n\to\infty$ and uniformly bounded periods $T_n$. By the Arzela-Ascoli
theorem, a subsequence of these periodic orbits converges to a
periodic orbit through $p$. This shows that the flow is also periodic
on $S$, so $S$ must be of one of the types (i)-(iii) described in (a)
of the lemma. 

Consider now case (b) in which the direction of $\bar X$ is
irrational. It follows from $X=\rho\bar X$ and~\eqref{eq:current} that
for $r>a$ the currents $c_r$ are given by
$$
   c_r(\alpha) = \int_M\alpha(X)d\fm_r
$$
for Borel measures $\fm_r$ supported on $\{r\}\times T^2$ and
invariant under the flow of $X$. Moreover, the total mass of these
measures is uniformly bounded away from $0$ and $\infty$. By the
Banach-Alaoglu theorem (viewing the space of Borel measures as dual to
the space of continuous functions with the weak$^*$ topology), there
exists a sequence $r_n\searrow a$ such that the $\fm_{r_n}$ converge
in the weak$^*$ topology to a nontrivial invariant Borel measure $\fm$
supported on $S$. Recall that weak$^*$ convergence means that for each
continuous $1$-form $\alpha$ defined on a neighbourhood of $S$ in
$V\cup S$ (and arbitrarily extended to $M$),
\begin{equation}\label{eq:fm}
   \lim_{n\to\infty}c_{r_n}(\alpha) = \int_M\alpha(X)d\fm. 
\end{equation}
Using this, we will now finish the proof in two steps. 

{\em Step 1. }
We will first show that $S_0$ cannot contain any periodic
orbit. Arguing by contradiction, suppose that $S_0$ carries a
periodic orbit $\delta_a$. Then by Proposition~\ref{prop:sing-level} a
tubular neighbourhood 
$(-\eps,\eps)\times S^1$ of $\delta_a=\{0\}\times S^1$ in $S_0$ is
foliated by periodic orbits $\{x\}\times 
S^1$. Pick a nondecreasing function $f:(-\eps,\eps)\to[-1,1]$ which
equals $\pm 1$ near $\pm\eps$ and has the regular level set
$\{0\}=f^{-1}(0)$. We extend $f$ by projection 
onto the first factor to a function on $(-\eps,\eps)\times S^1$, and
from there by another projection to a function on a neighbourhood
$[a,b)\times(-\eps,\eps)\times S^1$ in $V\cup S$ on which
$h(r,x,y)=h(r)$ and $f(r,x,y)=f(x,y)$. (Here we use
Lemma~\ref{lem:real-anal-sets}(c), which asserts that the level sets
of $h$ form a foliation near $\delta_a$.)  The exact $1$-form $df$ on
this set vanishes near $[a,b)\times\{\pm\eps\}\times S^1$, so
it can be extended by zero to a closed $1$-form $\alpha$ on
$h^{-1}([a,b))$. By construction we have $\alpha(X)\equiv
0$ on $S$ and hence $\int_M\alpha(X)d\fm=0$. 
On the other hand, the direction of $\bar X$ is irrational for $r>a$,
so we can arrange that the trajectories of $\bar X$ on $\{r\}\times
T^2$ are transverse to
the closed curve $\delta_r:=f^{-1}(0)\cap\{r\}\times T^2$ for all
$r>a$. Since $\alpha|_{\delta_r}=0$, it follows that
$\frac{1}{T}\int_0^T\alpha(\dot\gamma)dt=\frac{1}{T}\int_0^Tdf(\dot\gamma)dt
= \pm\frac{2}{T}$, where $T$ is the time
in which a trajectory $\gamma$ of $\bar X$ starting from 
$\delta_r$ hits it again for the first time. By averaging over the
trajectories we obtain $c_r(\alpha)=c\neq 0$ for all $r>a$, where the
constant $c$ does not depend on $r$ by closedness of $\alpha$ and
Lemma~\ref{lem:current}(c). But this contradicts the convergence
in~\eqref{eq:fm}. 

{\em Step 2. }
The preceding step excludes that $S$ contains a rational torus,
Klein bottle, or M\"obius strip, or cylinder with periodic flow. It
remains to exclude cylinders with nonperiodic flow. Arguing again by
contradiction, suppose that $Z\subset S$ is such a cylinder. 

We claim that the support of the measure $\fm$ does not meet $
\mathring{Z}=Z\setminus\p Z$. To see this, pick open neighbourhoods  
$U_\pm\subset M$ of the boundary orbits $\p_\pm Z$, and a compact
collar neighbourhood $W$ of $Z\setminus(U_+\cup U_-)$ in
$M\setminus(U_+\cup U_-)$ in which the flow moves from $U_-$ 
to $U_+$. Pick a flow box
$[a,b]\times[-\eps,\eps]\times[-\eps,\eps]\subset W$ with coordinates
$(r,x,y)$ around any point $p=(a,0,0)\in Z\setminus(U_+\cup U_-)$ on
which $h(r,x,y)=h(r)$ and  $X=\p_x$.  
Given a $1$-form $\alpha$ and the sequence $r_n\searrow a$ from above,
we pick $\bar X$-orbits $\gamma_n$ with 
$\gamma_n(0)=(r_n,0,0)$ and $\gamma_n(T_n)=(r_n,0,y_n)$ for some
$y_n\in[-\eps,\eps]$. This is possible because the flow on
$\{r_n\}\times T^2$ is irrational. Moreover, we pick the
$T_n\to\infty$ large enough so that 
$$
   \left|\frac{1}{T_n}\int_0^{T_n}\alpha\bigl(\dot\gamma_n(t)\bigr)dt -
   c_{r_n}(\alpha)\right| \to 0 
$$
as $n\to\infty$. It follows that
$$
   \int_M\alpha(X)d\fm = \lim_{n\to\infty}\frac{1}{T_n}\int_0^{T_n}
   \alpha\bigl(\dot\gamma_n(t)\bigr)dt.  
$$
Since the orbits $\gamma_n$ always traverse $W$ in some uniformly
bounded time, and the time they spend in $U_\pm$ tends to infinity as
$n\to\infty$ (because they get closer and closer to the periodic
orbits $\p_\pm Z$), the fraction of time the orbit $\gamma_n$ spends in $W$ 
converges to $0$ as $n\to\infty$. Hence the parts of $\gamma_n$
in $W$ do not contribute to the limit of the above integral, and thus
$\int_M\alpha(X)d\fm=0$ if $\alpha$ has support in $W$. Making $U_\pm$
arbitrarily small, this proves the claim.

Now we pick any simple closed curve $\delta_a\subset\mathring{Z}$
transverse to $X$ and construct a closed $1$-form $\alpha$ supported
near $\delta_a$ as in Step 1. As in Step 1, it follows that
$c_r(\alpha)=c\neq 0$ for all $r>a$ close to $a$. On the other hand, 
$\int_M\alpha(X)d\fm=0$ because the supports of $\alpha$ and $\fm$ do
not intersect. So again we have a contradiction to~\eqref{eq:fm},
which concludes the proof of Lemma~\ref{lem:ends}. 
\end{proof}

\begin{proof}[Proof of Proposition~\ref{prop:fol-cycle}]
{\em Case 1: The direction of $\bar X$ is rational.}
By Lemma \ref{lem:ends}(a), in this case the flow
of $\bar X$ is also periodic on $S$. Let $\gamma$ be any
periodic orbit in $S$ with covering number $d_\gamma$ as defined above. 
The local analysis near $\gamma$ in Lemma~\ref{lem:sing-orbit} shows
that $d_\gamma\gamma$ is homologous to the simple closed orbits on
$\{r\}\times T^2$ for $r>a$, Hence formula~\eqref{eq:c-rat},
closedness of $\nu$, and the normalization $\int_\gamma\nu=1/d_\gamma$
imply for $r>a$ close to $a$:  
$$
   c_-(\nu) = c_r(\nu) = \int_{d_\gamma\gamma}\nu = 1. 
$$ 
The same argument near $r=d$ yields $c_+(\nu)=1$. 

{\em Case 2: The direction of $\bar X$ is irrational. }
By Lemma~\ref{lem:ends}(b), $S_a=\ol{V}\cap h^{-1}(a)$ is either an
isolated periodic orbit or a torus on which the flow is irrational,
and the same holds for $S_b=\ol{V}\cap h^{-1}(b)$. If $S_a$ is an
irrational torus, there is another integrable region $V'\subset
M\setminus N$ containing $S_a$ in its boundary. If the direction of
the vector field $X$ on $V'$ is constant, then it must be irrational by
Lemma~\ref{lem:ends}(a). So the other boundary component of $V'$ is
again a torus with irrational flow or an isolated periodic orbit. 
Continuing like this in both directions, we find a maximal chain
of connected components of $M\setminus N$ on which the direction of
$X$ is constant and irrational, meeting along tori with irrational
flow. Now we use the remaining freedom in the normalization of $\nu$
on the isolated periodic orbits and irrational tori in
Remark~\ref{rem:nu} to arrange 
$c_+(\nu)=c_-(\nu)$ throughout this chain as follows: We arbitrarily
normalize $\nu$ on the first boundary component and then succesively 
extend it over the integrable regions, thus inducing normalizations on
the other boundary components such that the condition
$c_+(\nu)=c_-(\nu)$ holds across this chain. If the first and last
boundary components in this chain coincide (this can only happen when
both are irrational tori), then the $c_\pm(\nu)$ are the same at all
the boundary tori in this chain and thus match also at the ends. 


This concludes the proof of Proposition~\ref{prop:fol-cycle}, and
hence of Theorem~\ref{thm:real-anal-stable}. 
\end{proof}


\section{Construction of a smooth counterexample}\label{sec:counterex}

In this section we prove Proposition~\ref{prop:counterex} from the
Introduction by constructing a smooth counterexample. 

Consider the upper half plane $\HHH\subset \C$ with its standard
hyperbolic metric (of constant curvature $-1$) and the isometric
action of the M\"obius transformations $PSL(2,\R)$. Let $\Gamma\subset
PSL(2,\R)$ be a lattice so that $\Sigma=\HHH/\Gamma$ is a closed 
hyperbolic surface. Let $\lambda$ be the canonical contact $1$-form
(restriction of the Liouville form $p\,dq$) on the unit cotangent
bundle $S^*\Sigma$, whose Reeb vector field $R$ defines the geodesic
flow for the hyperbolic metric on $\Sigma$.  

{\bf Standardizing $\lambda$ near a closed Reeb orbit. }
We first derive a (partial) standard form for $\lambda$ near a closed
Reeb orbit $\gamma$ on $S^*\Sigma$. 

\begin{lemma}\label{lem:normform}
Each closed Reeb orbit $\gamma\subset S^*\Sigma$ possesses a tubular
neighbourhood $U\cong D\times S^1$ with coordinates $(x,y)$ in the
disk $D=\{x^2+y^2<\eps\}$ and $z$ in $S^1=\R/\Z$ along
$\gamma\cong\{0\}\times S^1$ in which $\lambda$ has the form
$$
   \lambda = (T_0+y^2-x^2)dz+\lambda_D,
$$
where $T_0=\int_\gamma\lambda>0$ and $\lambda_D$ is a $1$-form on $D$
(depending only on $x$ and $y$). 
\end{lemma}

\begin{proof} 
Consider the closed geodesic $\bar\gamma$ on $\Sigma$ corresponding to
the Reeb orbit $\gamma$. Let $\tilde\gamma$ be a lift of
$\bar\gamma$ to $\HHH$. 
The geodesic $\tilde\gamma\subset\HHH$ defines a $1$-parameter 
subgroup $G$ of the isometry group $PSL(2,\R)$ as follows. 
We pick a parametrization $\tilde\gamma(t)$ of $\tilde\gamma$ by
arclength and oriented according to $R$.  
For each $t\in\R$ there exists a unique $\phi_t\in PSL(2,\R)$
preserving $\tilde\gamma$ and mapping $\tilde\gamma(0)$ to
$\tilde\gamma(t)$. Since $\phi_{s+t}$ and $\phi_s\circ\phi_t$ both
preserve $\tilde\gamma$ and map $\gamma(0)$ to $\gamma(s+t)$, they are
equal, so $G:=\{\phi_t\mid t\in\R\}$ is a $1$-parameter subgroup of
$PSL(2,\R)$. 
The fact that 
$\tilde\gamma$ is the lift of a closed geodesic $\bar\gamma$ means that 
$$
   \Gamma_G:=\Gamma\cap G
$$
is a lattice in $G\cong\R$.

Since the projection $\HHH\to\Sigma$ is a covering and $\bar\gamma$ is
closed, each point on $\tilde\gamma$ has a neighbourhood that does not
intersect any other lift of $\bar\gamma$. By a compactness argument,
this implies that for $\eps>0$ suffiently small the strip 
$$
   \wt N:=\{p\in \HHH \mid \dist(p,\tilde\gamma)<\eps\}
$$
around $\tilde\gamma$ does not intersect any other lift of
$\bar\gamma$. After shrinking $\eps$ further, we may thus assume that
whenever any two points $p_1$ and $p_2$ in $\wt N$ are related by an
element $g\in \Gamma$ we have $g\in G$. Let the expressions
``to the left/right of $\tilde\gamma$'' have the obvious meaning. 
For $\delta\in (-\eps,\eps)$ we define the curves
$$
   \tilde\gamma_\delta:=\{p\in \HHH \mid \dist(p,\tilde\gamma)=|\delta| \text{
     and }p\text{ lies to the left/right of }\tilde\gamma\},
$$
where we choose ``left'' for $\delta<0$ and ``right'' for $\delta>0$,
and $\tilde\gamma_0:=\tilde\gamma$. 
Since $G$ acts by orientation preserving isometries, the curves 
$\tilde\gamma_\delta$ are orbits of $G$. Thus
$\wt N\cong(-\eps,\eps)\times\R$ is a $G$-invariant set on which
$G\cong\R$ acts by translation in the $\R$-direction. It follows that
the annulus 
$$
   N := \wt N/\Gamma= \wt N/\Gamma_G\cong (-\eps,\eps)\times
   S^1\subset \Sigma
$$
inherits an isometric action of $G/\Gamma_G\cong S^1$ by translations in
the $S^1$-direction. Consider the preimage $\pi^{-1}(N)\cong
(-\eps,\eps)\times S^1\times S^1$ of $N$ in the $S^1$-bundle
$\pi:S^*\Sigma\to\Sigma$. The isometric $S^1$-action on $N$ lifts (by
taking differentials) to an $S^1$-action on $\pi^{-1}(N)$ preserving
the canonical $1$-form $\lambda$. Since the action is free, there
exists a tubular neighbourhood $U\cong D\times S^1$ of
$\gamma\cong\{0\}\times S^1$ in $\pi^{-1}(N)$ with coordinates $(x,y)$
in the disk $D=\{x^2+y^2<\eps\}$ (for a different $\eps>0$) and $z$ in
$S^1=\R/\Z$ in which the action is given by translation in the
$S^1$-direction. Since the $1$-form $\lambda$ on $U$ is
$S^1$-invariant, we can write it uniquely as
$$
   \lambda=H(x,y)dz+\lambda_D
$$
for some function $H$ and some $1$-form $\lambda_D$ on $D$. Since
$\gamma$ is a Reeb orbit, the Reeb vector field $R$ along $\gamma$ is
given by $R(0,0,z)=T_0^{-1}\p_z$ for some constant $T_0>0$. The conditions
$\lambda(R)=1$ and $i_Rd\lambda=0$ along $\gamma$ now imply that 
$$
   H(0,0)=T_0,\qquad dH(0,0)=0.
$$
In particular, $\int_\gamma\lambda=T_0$. Now each closed orbit $\gamma$
for the geodesic flow on a hyperbolic surface is hyperbolic 
(see e.g.~\cite{HK95}). 
This translates into $H$ having a nondegenerate critical point of
index $1$ at $(0,0)$ (see the discussion in the next paragraph). So by
the Morse Lemma we can choose coordinates $(x,y)$ in which  
$$
   H(x,y)=T_0+y^2-x^2.
$$
\end{proof}

{\bf Reeb dynamics for $S^1$-invariant contact forms on $D\times
  S^1$. }
Consider more generally an $S^1$-invariant contact form on $U=D\times
S^1$ given by  
$$
   \lambda=H(x,y)dz+\lambda_D,\qquad H(x,y)>0,\qquad dH(0,0)=0.
$$
For sufficiently small $D$, the contact condition is equivalent to
$d\lambda_D$ being a positive area form on $D$. Since the Reeb vector
field $R$ is $S^1$-invariant, we can write it uniquely as 
$$
   R = T(x,y)^{-1}\p_z + R_D   
$$ 
for some positive function $T$ and some vector field $R_D$ on $D$. 
We compute
$$
   d\lambda = dH\wedge dz+d\lambda_D,\qquad i_Rd\lambda = -T^{-1}dH +
   i_{R_D}d\lambda_D,
$$
so the condition $i_Rd\lambda=0$ translates into
$$
   dH = T\,i_{R_D}d\lambda_D. 
$$
In other words, $R_D$ is the Hamiltonian vector field of the
autonomous Hamiltonian function $H$ with respect to the symplectic
form $-T\,d\lambda_D$ on $D$. Writing $w=(x,y)\in D$, the equations for the
Reeb flow become
$$
   \dot w=R_D(w),\qquad \dot z=T(w)^{-1}. 
$$
Thus the Reeb flow projects onto the Hamiltonian flow of $H$ on $D$,
and the $z$-component can be integrated to
$z(t)=z(0)+T(w)^{-1}t$. Since $z$ lives in $\R/\Z$, we see that the
Reeb orbit starting at time $t=0$ at $(w,0)$ returns to the slice
$D\times\{0\}$ at time $T(w)$ and the Poincar\'e return map on
$D\times\{0\}$ is given by
$$
   \phi(w) = \psi_{T(w)}(w), 
$$
where $\psi_t$ is the flow of $R_D$ on $D$ (which may run out of $D$)
and $T(w)$ is the return time. Note that $dH(0,0)=0$ implies
$R_D(0,0)=0$, so $(0,0)\in D$ is a fixed point of $\phi$. The linearization
of $\phi$ at $(0,0)$ is given by
$$
   D\phi(0,0)\cdot w = D\psi_{T_0}(0,0)\cdot w + (DT(0,0)\cdot w)\,R_D(0,0) =
   D\psi_{T_0}(0,0)\cdot w,
$$
where $T_0:=T(0,0)=H(0,0)>0$. Thus $D\phi(0,0) = D\psi_{T_0}(0,0)$ is the time
$T_0$ map of the linearized Hamiltonian flow
$$
  \dot w = JS\cdot w,\qquad J=\begin{pmatrix}0 & -1 \\ 1 & 0\end{pmatrix}
$$
where $S$ is the Hessian of $H$ at $(0,0)$. A short computation shows
that the eigenvalues $\tau$ of $JS$ satisfy $\tau^2+\det S=0$. Hence the
Reeb orbit $\gamma=\{(0,0)\}\times S^1$ is
\begin{itemize}
\item degenerate (at least one eigenvalue is zero) iff $\det S=0$, i.e., $(0,0)$ is
  a degenerate critical point of $H$;
\item hyperbolic (both eigenvalues are real and nonzero) iff $\det S<0$, i.e., $(0,0)$ is
  a nondegenerate critical point of $H$ of index $1$;
\item elliptic (both eigenvalues are imaginary and nonzero) iff $\det S>0$, i.e., $(0,0)$ is
  a nondegenerate critical point of $H$ of index $0$ or $2$. 
\end{itemize} 
This fills in the argument used in the proof of
Lemma~\ref{lem:normform}. For later use, let us record the following
consequences of our discussion:
\begin{enumerate}
\item Closed Reeb orbits in $U$ are of the form $\{w\}\times S^1$ for
  critical points $w$ of $H$.
\item  Invariant tori for the Reeb flow in $U$ are of the form
  $H^{-1}(c)\times S^1$, where the level sets $H^{-1}(c)$ of $H$ are diffeomorphic to
  the circle. 
\end{enumerate}
In particular, if $\gamma=\{(0,0)\}\times S^1$ is hyperbolic and $D$
sufficiently small, then $U$ contains no closed orbits except $\gamma$
and no invariant tori. 
\medskip

{\bf Modifying $\lambda$ near a closed Reeb orbit. }
Now we turn back to the neighbourhood $U$ of a closed Reeb orbit
$\gamma$ described in Lemma~\ref{lem:normform}. 
Let $\chi:[0,\eps]\to[0,1]$ 
be a monotone cutoff function which equals
$0$ near $0$ and $1$ near $\eps$.
Set $r^2:=x^2+y^2$. We modify the contact form $\lambda$ on $U$ to
\begin{equation}\label{eq:tilde}
   \lambda_\chi:=H_\chi(x,y)dz+\lambda_D,\qquad
   H_\chi(x,y):=T_0+\chi(r^2)(y^2-x^2). 
\end{equation}
By the discussion above, $\lambda_\chi$ is a contact form. Since it
agrees with $\lambda$ near the boundary of $U$, we can extend it by
$\lambda$ to a contact form on $S^*\Sigma$ that we will still denote
by $\lambda_\chi$. 
Note that on the region $\{\chi=0\}$ the form $\lambda_\chi$ agrees
with $T_0dz+\lambda_D$, so its Reeb flow is periodic moving in the
$z$-direction (in particular, it is integrable) on this region. On the
other hand, we have

\begin{lemma}\label{lem:non-int}
Every smooth integral of motion for the Reeb flow of $\lambda_\chi$ on
$S^*\Sigma\setminus\{\chi=0\}$ is constant. 
\end{lemma}

\begin{proof}
Let us first determine the critical points of $H_\chi$ on $D$. At a
critical point we have
\begin{align*}
   0 = \p_xH_\chi &= 2x\chi'(r^2)(y^2-x^2)-2x\chi(r^2), \cr
   0 = \p_yH_\chi &= 2y\chi'(r^2)(y^2-x^2)+2y\chi(r^2).
\end{align*}
Subtracting $x/2$ times the first equation from $y/2$ times the second
we obtain
$$
   0 = (y^2-x^2)^2\chi'(r^2) + r^2\chi(r^2). 
$$
Since all terms on the right hand side are nonnegative, this is only
possible for $\chi(r^2)=0$. So the only critical level set of $H_\chi$
is $\{H_\chi=T_0\}$. It consists of the disk $\{\chi=0\}$ of critical
points together with the $4$ lines $\{x=\pm y\}\setminus\{\chi=0\}$ of
regular points. Since $H_\chi$ agrees with $H$ near $\p D$, we obtain
a one-to-one correspondence between the level sets of $H$ and those of
$H_\chi$, the only difference being that the critical point of $H$ at
the origin has been replaced by the critical disk $\{\chi=0\}$ for
$H_\chi$. In view of consequences (i) and (ii) above, this implies a
one-to-one correspondence between invariant tori of $H$ and those
of $H_\chi$ outside the disk $\{\chi=0\}$. 

Suppose now that $f:S^*\Sigma\setminus\{\chi=0\}\to\R$ is a
nonconstant smooth integral of motion for the Reeb flow of
$\lambda_\chi$. By the Arnold-Liouville
theorem~\cite{Arn66}, this gives us an open region in
$S^*\Sigma\setminus\{\chi=0\}$ that is foliated by invariant $2$-tori
for the Reeb flow of $\lambda_\chi$. By the preceding discussion, this
gives rise to an open region in $S^*\Sigma\setminus\{\chi=0\}$
foliated by invariant $2$-tori for the Reeb flow of $\lambda$. But
this contradicts ergodicity of the geodesic flow on $S^*\Sigma$ 
(see e.g.~\cite{HK95}),
hence such an integral of motion cannot exist. 
\end{proof}

{\bf Modifying $\lambda$ near $4$ closed Reeb orbits. }
Now we perform the preceding construction at $4$ closed Reeb orbits as
follows. Suppose that $\Sigma$ has genus at least $3$. 
Pick four disjoint simple closed geodesics $\bar\gamma_i$
on $\Sigma$ with corresponding closed Reeb orbits $\gamma_i$ and
periods $T_i=\int_{\gamma_i}\lambda$, $i=1,\dots,4$. (Their precise
choice will be fixed later.) 
Let $U_i\cong D_i\times S^1$ be tubular neighbourhoods of $\gamma_i$ as
in Lemma~\ref{lem:normform} (with $T_0$ replaced by $T_i$ and some
$\eps_i>0$). We modify $\lambda$ as described in~\eqref{eq:tilde} on
each $U_i$, using some cutoff functions $\chi_i$, to obtain a new
contact form $\tilde\lambda$ on $S^*\Sigma$. (We could actually
take the same $\eps>0$ and the same cutoff function $\chi$ for all
$i$, but it will not matter for the following construction.)  
Recall that in the canonical coordinates $(x,y,z)$ on each region
$$
   V_i:=\{\chi_i=0\}\subset U_i
$$ 
the form $\tilde\lambda$ agrees with
$T_idz+\lambda_{D_i}$ for some $1$-form $\lambda_{D_i}$ on the disk
$D_i=\{x^2+y^2<\eps_i\}$ such that $d\lambda_{D_i}$ is a positive area
form. 
After a modification of $\lambda_{D_i}$ on $\{\chi_i=0\}\subset D_i$, keeping it
fixed near the boundary of $\{\chi_i=0\}$, we may assume that 
$$
   d\lambda_{D_i} = dr\wedge d\phi
$$
in polar coordinates $(r,\phi)$ corresponding to the Cartesian
coordinates $(x,y)$ on some annulus 
$$
   A_i:=\{r_i\le r\le r_i+\delta\}\subset \{\chi_i=0\}\subset D_i.
$$
(Note that we use the same $\delta>0$ for all $i=1,\dots,4$.)
We keep denoting the resulting contact form by $\tilde\lambda$. Note
that this modification of $\lambda_{D_i}$ does not change the Reeb
vector field $\tilde R$ of $\tilde\lambda$, which is still given by
$T_i^{-1}\p_z$ on $V_i$. Recall that a {\em stabilizing $1$-form} for
the vector field $\tilde R$ is a $1$-form $\nu$ satisfying $i_{\tilde
  R}d\nu=0$ and $\nu(\tilde R)>0$. The following lemma shows that such
forms must be very special.  

\begin{lemma}\label{lem:c}
Let $\nu$ be a stabilizing $1$-form for $\tilde R$ on
$S^*\Sigma\setminus(V_1\cup\cdots\cup V_4)$. Then 
$$
   \nu = c\tilde\lambda + \beta
$$
for some constant $c\ne 0$ and some closed $1$-form $\beta$ on
$S^*\Sigma\setminus(V_1\cup\cdots\cup V_4)$. 
\end{lemma}

\begin{proof}
Since $i_{\tilde R}d\nu=i_{\tilde R}d\tilde\lambda=0$, we can write
$d\nu=f\,d\tilde\lambda$ for a smooth function on
$S^*\Sigma\setminus(V_1\cup\cdots\cup V_4)$ which is invariant
under the flow of $\tilde R$. By Lemma~\ref{lem:non-int} (which
clearly continues to hold for $\lambda$ being modified near $4$ orbits
instead of just one), we must have $f\equiv c$ for some constant
$c\in\R$. It follows that $\beta:=\nu-c\,\tilde\lambda$ is closed. 

It remains to rule out the case $c=0$. 
Since $\Sigma$ has genus at least $3$, we find a simple
closed geodesic $\bar\gamma$ that is disjoint from the simple closed
geodesics $\bar\gamma_1,\dots,\bar\gamma_4$. Let $\gamma,\gamma'$ be
the closed Reeb orbits corresponding to $\bar\gamma$ and to
$\bar\gamma$ parametrized backwards, respectively. Rotating at each
point of $\bar\gamma$ the unit tangent vector to $\bar\gamma$ to its
opposite yields a cylinder in $S^*\sigma$ connecting $\gamma'$ to 
$-\gamma$ (the curve $\gamma$ oppositely oriented). Since the whole
cylinder projects onto
$\bar\gamma\subset\Sigma\setminus(\bar\gamma_1\cup\dots\cup\bar\gamma_4)$,
this shows that 
$$
   [\gamma']=-[\gamma]\in 
   H_1\Bigl(S^*\Sigma\setminus(\gamma_1\cup\dots\cup\gamma_4)\Bigr) \cong 
   H_1\Bigl(S^*\Sigma\setminus(V_1\cup\dots\cup V_4)\Bigr).
$$
Now if $c=0$, then $\nu=\beta$ would be closed and thus
$\int_{\gamma'}\nu=-\int_\gamma\nu$, contradicting the condition
(which follows from $\nu(\tilde R)>0$) that $\int_{\gamma'}\nu$ and
$\int_\gamma\nu$ must both be positive. 
\end{proof}

{\bf Cut and paste}
Now we will build a new manifold $M$ out of $S^*\Sigma$ by a cut and
paste construction. Let us represent a genus $5$ surface $\Sigma$ as
an iterated connected sum of $5$ tori in linear order (in the obvious
notation)  
\begin{align*}
   \Sigma := \Bigl(T^2\setminus D_1^+\Bigr)
   &\cup_{\bar\gamma_1}\Bigl(T^2\setminus(D_1^-\cup D_2^+)\Bigr) 
   \cup_{\bar\gamma_2}\Bigl(T^2\setminus(D_2^-\cup D_3^+)\Bigr) \cr 
   &\cup_{\bar\gamma_3}\Bigl(T^2\setminus(D_3^-\cup D_4^+)\Bigr) 
   \cup_{\bar\gamma_4}\Bigl(T^2\setminus D_4^-\Bigr). 
\end{align*}
Note that the gluing curves $\bar\gamma_i$,
$i=1,\dots,4$, on $\Sigma$ are homologous and appear as boundary
curves in a pair-of-pants decomposition of $\Sigma$. In view of the
Fenchel-Nielsen coordinates on Teichm\"uller space 
(see e.g.~\cite{Hu97}),
there exist hyperbolic metrics on $\Sigma$ making the $\bar\gamma_i$
geodesics with arbitrarily prescribed positive values of their lengths
$T_1,\dots,T_4$. We choose a hyperbolic metric such that
\begin{equation}\label{eq:T-diff}
   T_2-T_1 \neq T_4-T_3. 
\end{equation}
Note that the lengths $T_i$ are equal to the periods
$\int_{\gamma_i}\lambda$ of the corresponding closed Reeb orbits
$\gamma_i$ in $S^*\Sigma$. 
We modify $\lambda$ near $\gamma_1,\dots,\gamma_4$ to the contact from
$\tilde\lambda$ as above. Recall that 
$$
   d\tilde\lambda = d\lambda_{D_i} = dr\wedge d\phi
$$ 
in coordinates $(r,\phi,z)$ on the regions 
$$
   A_i\times S^1\subset V_i,\qquad A_i:=\{r_i\le r\le
   r_i+\delta\}\subset \{\chi_i=0\}\subset D_i.
$$
We cut out from $S^*\Sigma$ the four solid tori $\{r < r_i\}\subset
V_i$. Then we glue the collar neighbourhoods of the boundary tori of the
resulting manifold pairwise via the orientation preserving
diffeomorphism
$$
   \Phi_{12}:[r_1,r_1+\delta]\times S^1\times S^1\to[r_2,r_2+\delta]\times
   S^1\times S^1,\quad
   (r,\phi,z)\mapsto (r_2+\delta+r_1-r,-\phi,z),
$$
and similarly for $i=3,4$. This yields an oriented, closed manifold
$M$. Since the gluing maps preserve the area form $dr\wedge d\phi$,
the $2$-form $d\tilde\lambda$ descends to a nowhere vanishing closed
$2$-form $\om$ on $M$. 
Now the following lemma concludes the proof of
Proposition~\ref{prop:counterex}. 

\begin{lemma}\label{lem:mainres}
There exists a nowhere vanishing vector field $X$ on the above
manifold $M$ generating $\ker\om$ with the following properties: 

(i) $X$ admits no stabilizing $1$-form (note that this is actually a
property of $\om$);

(ii) $X$ solves the stationary Euler
equations~\eqref{eq1},~\eqref{eq2} for some metric and volume form on
$M$.  
\end{lemma}

\begin{proof}
We will repeatedly use the following simple observation:
\begin{itemize}
\item[(O)] Let $\nu$ be
a $1$-form stabilizing a vector field $X$ generating the foliation by
circles $\{p\}\times S^1$ on a connected manifold $P\times S^1$ (in
our applications $P$ will be a disk or annulus). Then
$\int_{\{p\}\times S^1}\nu$ does not depend on $p\in P$. 
\end{itemize}
To see this,
connect two points $p,q\in P$ by a curve $\gamma$. The condition
$i_Xd\nu=0$ and Stokes' theorem now imply 
$$
0=\int_{\gamma\times
  S^1}d\nu = \int_{\{q\}\times S^1}\nu - \int_{\{p\}\times S^1}\nu.
$$ 
We first apply this observation to the solid tori $V_i\cong \{\chi_i=0\}\times
S^1$ with the contact form $\tilde\lambda=T_idz+\lambda_{D_i}$. Let us
pick points $w_i\in\p D_i$. Then the closed Reeb orbits
$\gamma_i':=\{w_i\}\times S^1$ satisfy
\begin{equation}\label{eq:T}
   \int_{\gamma_i'}\tilde\lambda = \int_{\gamma_i}\tilde\lambda =
   T_i. 
\end{equation}
Let now $(M,\om)$ be as in the lemma, and $X$ be a nowhere vanishing
vector field generating $\ker\om$. Consider in $M$ the regions
\begin{align*}
   M_{12} &:= (V_1\setminus\{r <
   r_1\})\cup_{\Phi_{12}}(V_2\setminus\{r<r_2\}) \cong [1,2]\times
   S^1\times S^1, \cr
   M_{34} &:= (V_3\setminus\{r <
   r_3\})\cup_{\Phi_{34}}(V_4\setminus\{r<r_4\}) \cong [3,4]\times
   S^1\times S^1. 
\end{align*}
By construction, $X$ generates the foliation by circles
$\{(r,\phi)\}\times S^1$ in the $z$-direction on $M_{12}$ and
$M_{34}$. Note that 
$$
   \Bigl(M\setminus(M_{12}\cup M_{34}),\om\Bigr) \cong
   \Bigl(S^*\Sigma\setminus(V_1\cup\cdots\cup V_4), d\tilde\lambda
   \Bigr).  
$$
In particular, the closed Reeb orbits $\gamma_i'$ defined above
can be viewed as sitting on the boundary components $\{i\}\times
S^1\times S^1$, $i=1,\dots,4$, of $M_{12}$ resp.~$M_{34}$. 

Now we can prove (i). Recall from the
construction of $\Sigma$ that $\bar\gamma_2-\bar\gamma_1$ is the
boundary of a region $\Sigma_{12}\subset\Sigma$ diffeomorphic to a
$2$-torus with two disks removed. In particular, $\bar\gamma_1$ and
$\bar\gamma_2$ are homologous in $\Sigma$. Their lifts to
$S^*\Sigma$ satisfy
\begin{equation}\label{eq:hom1}
   [\gamma_2] = [\gamma_1] - 2[F] \in H_1(S^*\Sigma),
\end{equation}
where $[F]$ is the class of a fibre of the circle bundle
$S^*\Sigma\to\Sigma$ and $-2$ is the Euler characteristic of
$\Sigma_{12}$. The analogous argument applied to the region
$\Sigma_{34}\subset\Sigma$ bounded by $\bar\gamma_4-\bar\gamma_3$
yields 
\begin{equation}\label{eq:hom2}
   [\gamma_4] = [\gamma_3] - 2[F] \in H_1(S^*\Sigma).
\end{equation}
(The precise coefficient $-2$ in front of $[F]$ in~\eqref{eq:hom1}
and~\eqref{eq:hom2} will not matter for us, but it will be important
that the coefficient is the same in both equations.) 
Since the regions $\Sigma_{12}$ and $\Sigma_{34}$ are disjoint, 
relation~\eqref{eq:hom1} continues to hold in
$H_1\bigl(S^*\Sigma\setminus(\gamma_3\cup\gamma_4)\bigr)$, and
relation~\eqref{eq:hom2} in
$H_1\bigl(S^*\Sigma\setminus(\gamma_1\cup\gamma_2)\bigr)$. Replacing
the $\gamma_i$ by their push-offs $\gamma_i'$, we obtain the
relations 
\begin{equation}\label{eq:hom3}
   [\gamma_2'] - [\gamma_1'] + 2[F] = [\gamma_3'] - [\gamma_4'] + 2[F] = 0
   \in H_1\Bigl(S^*\Sigma\setminus(V_1\cup\cdots\cup V_4)\Bigr). 
\end{equation}
Suppose now that $\nu$ is a $1$-form on $M$
stabilizing $X$. By Lemma~\ref{lem:c}, on \\
$S^*\Sigma\setminus(V_1\cup\cdots\cup V_4)$ we have 
$$
   \nu = c\tilde\lambda + \beta
$$
for some constant $c\neq 0$ and some closed $1$-form $\beta$ on
$S^*\Sigma\setminus(V_1\cup\cdots\cup V_4)$. Integrating $\nu$ over
$\gamma_1',\dots,\gamma_4'$ and using~\eqref{eq:hom3}
and~\eqref{eq:T}, we obtain  
\begin{equation}\label{eq:int1}
   \int_{\gamma_2'}\nu - \int_{\gamma_1'}\nu = c(T_2-T_1) -
   2\int_F\beta, 
\end{equation}
\begin{equation}\label{eq:int2}
   \int_{\gamma_4'}\nu - \int_{\gamma_3'}\nu = c(T_4-T_3) -
   2\int_F\beta.  
\end{equation}
Due to condition~\eqref{eq:T-diff} on the $T_i$ (and since $c\neq 0$),
the right hand sides of~\eqref{eq:int1} and~\eqref{eq:int2} are not
equal. On the other hand, applying the observation (O) above to the regions
$M_{12}$ and $M_{34}$, we conclude that the left hand sides
of~\eqref{eq:int1} and~\eqref{eq:int2} are both zero. This
contradiction proves (i). 

For (ii), we need to construct a nowhere vanishing vector field $X$ on 
the manifold $M$ generating $\ker\om$ such that 
\begin{equation}\label{eq:Euler-again}
   i_Xd\lambda=-dh,\qquad \lambda(X)>0
\end{equation} 
for some $1$-form $\lambda$ and function $h$ on
$M$. (One should not confuse $\lambda$ with the canonical contact form
on $S^*\Sigma$, which will not be used any more. It was explained in
the Introduction how to recover from these data a volume form and
metric for which $X$ satisfies the stationary Euler
equations~\eqref{eq1},~\eqref{eq2}.) 

On $M\setminus(M_{12}\cup M_{34})$ we take $\lambda:=\tilde\lambda$,
$X:=\tilde R$ its Reeb vector field, and $h:=0$, so $i_Xd\lambda=0$ and
$\lambda(X)=1$ on this region. Note that we can extend $(\lambda,X,h)$
slightly into $M_{12}$ and $M_{34}$ by the same formulas (we can 
actually extend them up to the region where the gluing happens).  
Recall that the kernel foliation of $\om$ on the region
$M_{12}\cong[1,2]\times S^1\times S^1$ consists of the circles
$\{(r,\phi)\}\times S^1$ in the $z$-direction. By
construction of $\tilde\lambda$, the given data are equal to  
$$
   \lambda=T_idz+\lambda_{D_i},\qquad X=T_i^{-1}\p_z,\qquad h=0
$$
near the boundary component $\{i\}\times T^2$, $i=1,2$. We extend
$\lambda$ and $X$ over $M_{12}$ satisfying~\eqref{eq:Euler-again} as
follows (the extension over $M_{34}$ is analogous). 

Let $b$ be a positive function on $[1,2]$ which is constant $T_1$ near
$1$, constant $T_2$ near $2$, and has regions 
with positive derivative as well as regions with negative
derivative. Let $g$ be a positive function on $[1,2]$ which is constant
$T_1^{-1}$ near $1$, constant $T_2^{-1}$ near $2$, and such that  
\begin{equation}\label{eq:int}
   \int_1^2g(r)b'(r)dr=0.
\end{equation}
(This is possible because $b'$ changes signs.) We pick any $1$-form
$\lambda_D$ on $[1,2]\times S^1$ (depending only on $r$ and $\phi$)
which agrees with $\lambda_{D_i}$ near $\{i\}\times S^1$, $i=1,2$.  
Now we define the extension over $[1,2]\times T^2$ by
$$
   \lambda:=b(r)dz+\lambda_D,\qquad X:=g(r)\p_z,\qquad
   h(r):=\int_1^rg(s)b'(s)ds. 
$$
The choice of $b$ and $g$ ensures that this matches the given data
near the boundary (condition~\eqref{eq:int} ensures that $h(r)=0$ near
$r=2$). Note that $X$ generates $\ker\om$. Since  
$$
   i_Xd\lambda = i_{g(r)\p_z}b'(r)dr\wedge dz+i_{g(r)\p_z}b'(r)d\lambda_D=-g(r)b'(r)dr+0 = -dh
$$
and $\lambda(X) = b(r)g(r)>0$, the triple $(\lambda,X,h)$
satisfies~\eqref{eq:Euler-again}. This concludes the proof of
Lemma~\ref{lem:mainres}, and hence of
Proposition~\ref{prop:counterex}. 
\end{proof}


\end{document}